\documentclass[12pt]{amsart}
\usepackage{pstricks}
\usepackage{fullpage}
\usepackage{amsfonts}
\usepackage{amssymb}
\usepackage{amsthm}
\usepackage{latexsym,amsmath}
\usepackage{graphicx}
\newcommand{\bib}{\clearpage
\addtocounter{page}{1}
\addcontentsline{toc}{chapter}{\numberline{}\bibname}
\addtocounter{page}{-1}
\bibliographystyle{plain}
\bibliography{biblio14}}

\newtheorem{theorem}{Theorem}[section]
\newtheorem{Definition}[theorem]{Definition}

\newenvironment{theorem*}[1]{\medskip
                            \noindent
                            {\bf Theorem #1. }\ %
                            \begingroup \sl}
                            {\endgroup\medskip}

\usepackage{color}

\title{A Consistency Proof for Some Restrictions of Tait's Reflection Principles}

\author[$\mathrm{M^{\lowercase{c}}Callum}$]{Rupert $\mathrm{M^{\lowercase{c}}Callum}$}

\begin{document}

\maketitle

\begin{abstract}

In \cite{Tait2005a} Tait identifies a set of reflection principles which he calls $\Gamma^{(2)}_{n}$-reflection principles which Peter Koellner has shown to be consistent relative to $\kappa(\omega)$, the first $\omega$-Erd\"os cardinal, in \cite{Koellner2003}. Tait also goes on in the same work to define a set of reflection principles which he calls $\Gamma^{(m)}_{n}$-reflection principles; however Koellner has shown that these are inconsistent when $m>2$ in \cite{Koellner2009}, but identifies restricted versions of them which he proves consistent relative to $\kappa(\omega)$. In this paper we introduce a new large-cardinal property with an ordinal parameter $\alpha>0$, calling those cardinals which satisfy it $\alpha$-reflective cardinals. Its definition is motivated by the remarks Tait makes in \cite{Tait2005a} about why reflection principles must be restricted when parameters of third or higher order are introduced. We prove that if $\kappa$ is $\beth_{\kappa+\alpha+1}$-supercompact and $0<\alpha<\kappa$ then $\kappa$ is $\alpha$-reflective. Furthermore we show that $\alpha$-reflective cardinals relativize to $L$, and that if $\kappa(\omega)$ exists then the set of cardinals $\lambda<\kappa(\omega)$, such that $\lambda$ is $\alpha$-refective for all $\alpha$ such that $0<\alpha<\lambda$, is a stationary subset of $\kappa(\omega)$. We show that an $\omega$-reflective cardinal satisfies some restricted versions of $\Gamma^{(m)}_{n}$-reflection, as well as all the reflection properties which Koellner proves consistent in \cite{Koellner2009}.

\end{abstract}

\thanks{Part of this paper was written while I was a Research Intensive Academic at the Australian Catholic University. I made further revisions to it while I was an Adjunct Lecturer at the University of New South Wales, and while I held a post-doctoral position at the University of M\"unster. I am grateful to these institutions for their support. I am also thankful to Peter Koellner for giving me helpful feedback on an early draft.}

\section{Introduction}

We are going to investigate reflection principles, which postulate the existence of a level of the universe $V_{\kappa}$, whose properties reflect down to some lower level $V_{\beta}$ where $\beta<\kappa$. It is useful to begin by considering reflection principles involving second-order parameters only. In later sections we will consider the issues which arise when one introduces higher-order parameters.

\bigskip

The cardinals yielded by these reflection principles involving second-order parameters only are called ``indescribable cardinals". These principles assert the existence of a cardinal $\kappa$ such that certain statements true in $V_{\kappa}$ hold when relativized to a level $V_{\beta}$ where $\beta<\kappa$. The strength of the reflection principles increase as one increases the expressive power of the language in which the statements are formulated, and the complexity of the formulas which express them. For example, one may consider the case where the language $\textsf{L}$ in which the statements are expressed is the union of the $n$th-order languages of set theory for all $n<\omega$. We denote the order of a variable with a superscript, so that $X^{(m)}$ is a variable of $m$th order. Second-order variables range over classes, for example. If a formula $\varphi$ in the language $\textsf{L}$ is relativized to $V_{\kappa}$, then the variables of $m$th order range over $V_{\kappa+m-1}$. Here $V_{\kappa}$ is being treated as the universe.

\begin{Definition} We say that a formula in the language $\textsf{L}$ is a $\Pi^{m}_{0}$-formula if the only quantified variables it contains are at most $m$th order. \newline We say that a formula in the language $\textsf{L}$ is a $\Pi^{m}_{1}$-formula if it consists of a block of universal $(m+1)$th order quantifiers tacked on to the beginning of a $\Pi^{m}_{0}$-formula. \newline We say that a formula in the language $\textsf{L}$ is $\Sigma^{m}_{k+1}$ if it consists of a block of existential $(m+1)$th-order quantifiers tacked on to the beginning of a $\Pi^{m}_{k}$-formula. \newline We say that a formula in the language $\textsf{L}$ is $\Pi^{m}_{k+1}$ if it consists of a block of universal $(m+1)$th-order quantifiers tacked on to the beginning of a $\Sigma^{m}_{k}$-formula. \end{Definition}

\begin{Definition} If $\varphi$ is formula in the language $\textsf{L}$, we denote by $\varphi^{\beta}$ the result of relativizing every $m$th-order quantifier to $V_{\beta+m-1}$. If $X^{(2)}$ is a second-order variable we abbreviate $X^{(2)}\cap V_{\beta}$ to $X^{(2),\beta}$. \end{Definition}

\begin{Definition} If $\Omega$ is a class of formulas of $\textsf{L}$, we say that $\kappa$ is $\Omega$-indescribable if for all formulas $\varphi\in\Omega$ whose only free variable is second-order, for all sets $U\subset V_{\kappa}$, $\varphi^{\kappa}(U)\Longrightarrow\exists\beta<\kappa$ $\varphi^{\beta}(U^{\beta})$. We say that $\kappa$ is totally indescribable if it is $\Pi^{m}_{n}$-indescribable for all $m, n>0$. \end{Definition}

\begin{Definition} Suppose that $\alpha$ is an ordinal. We say that $\kappa$ is $\alpha$-indescribable if for all $\Pi^{1}_{0}$ formulas $\varphi$ in the language $\textsf{L}$ whose only free variable is second-order, for all sets $U\subset V_{\kappa}$, $V_{\kappa+\alpha}\models \varphi(U)\Longrightarrow \exists\beta<\kappa$ $V_{\beta+\alpha}\models \varphi(U^{\beta})$ for some $\beta<\kappa$. (This definition is due to Jensen.) \end{Definition}

\begin{Definition} We say that $\kappa$ is absolutely indescribable if $\kappa$ is $\alpha$-indescribable for all $\alpha<\kappa$. (This definition appears in \cite{Tait2005a}.) \end{Definition}

\begin{Definition} We say that $\kappa$ is extremely indescribable if for all formulas $\Pi^{1}_{0}$ formulas $\varphi$ in the language $\textsf{L}$ whose only free variable is second-order, for all sets $U\subset V_{\kappa}$, $V_{\kappa+\kappa}\models\varphi(U)\Longrightarrow\exists \beta<\kappa$ $V_{\beta+\beta}\models\varphi(U^{\beta})$. (This definition is due to Harvey Friedman.) \end{Definition}

Here we are giving examples of cardinals $\kappa$ such that $V_{\kappa}$ satisfies reflection of formulas with second-order parameters. Let us next consider what happens when we move to parameters of third or higher order.

\section{Reflection involving parameters of third or higher order}

We have already defined $A^{(2),\beta}$ when $A^{(2)}$ is a second-order parameter. We define $A^{(m+1),\beta}=\{B^{(m),\beta}\mid B^{(m)}\in A^{(m+1)}\}$ for all integers $m\geq 2$. We say that $\kappa$ satisfies reflection with $m$th-order parameters for all formulas in a class $\Omega$ if, whenever $\varphi^{\kappa}(U^{(m)})$ for some $U^{(m)}\in V_{\kappa+m-1}$, there exists a $\beta<\kappa$ such that $\varphi^{\beta}(U^{(m),\beta})$. It is inconsistent to postulate the existence of cardinal $\kappa$ which satisfies reflection for all formulas in $\textsf{L}$, where $\textsf{L}$ is the language defined in the second paragraph of the Introduction, with third-order parameters. This was observed by Reinhardt in \cite{Reinhardt74}.

\bigskip

To see this, let $X^{(3)}$ be a third-order variable and let $\varphi$ be the assertion that every element of $X^{(3)}$ is a bounded subset of $\mathrm{On}$. This assertion can be written as a sentence in $\textsf{L}$ with $X^{(3)}$ as the only free variable. Now, suppose that $\kappa$ satisfies reflection for such sentences with third-order parameters. Let $U^{(3)}=\{\{\xi\mid\xi<\alpha\}\mid\alpha\in\mathrm{On}\cap\kappa\}$. We have $\varphi^{\kappa}(U^{(3)})$. So by the hypothesis about $\kappa$ we must have $\varphi^{\beta}(U^{(3),\beta})$ for some $\beta<\kappa$. But this is impossible because $U^{(3),\beta}$ contains the set $\{\xi\mid\xi<\beta\}$, which is not bounded in $\mathrm{On}\cap V_{\beta}$. Thus no ordinal $\kappa$ satisfies reflection for formulas in $\textsf{L}$ with third-order parameters.

\bigskip

This means that in order to formulate consistent reflection principles for formulas with third-order parameters or higher one must constrain the formulas relativized in some way. Let us consider what Tait writes in \cite{Tait2005a} about this issue.

\bigskip

``One plausible way to think about the difference between reflecting $\varphi(A)$ when $A$ is second-order and when it is of higher-order is that, in the former case, reflection is asserting that, if $\varphi(A)$ holds in the structure $\langle R(\kappa),\in,A \rangle$, then it holds in the substructure $\langle R(\beta),\in, A^{\beta} \rangle$ for some $\beta<\kappa \ldots$ But, when $A$ is higher-order, say of third-order this is no longer so. Now we are considering the structure $\langle R(\kappa), R(\kappa+1), \in, A \rangle$ and $\langle R(\beta), R(\beta+1), \in, A^{\beta} \rangle$. But, the latter is not a substructure of the former, that is the `inclusion map' of the latter structure into the former is no longer single-valued: for subclasses $X$ and $Y$ of $R(\kappa), X\neq Y$ does not imply $X^{\beta}\neq Y^{\beta}$. Likewise for $X\in R(\beta+1), X\notin A$ does not imply $X^{\beta}\notin A^{\beta}$. For this reason, the formulas that we can expect to be preserved in passing from the former structure to the latter must be suitably restricted and, in particular, should not contain the relation $\notin$ between second- and third-order objects or the relation $\neq$ between second-order objects."

\bigskip

Now, suppose that we are reflecting a formula $\varphi$ of the form

\bigskip

\noindent $\forall X_{1}^{(m_{1})} \exists Y_{1}^{(n_{1})} \forall X_{2}^{(m_{2})} \exists Y_{2}^{(n_{2})} \cdots \forall X_{k}^{(m_{k})} \exists Y_{k}^{(n_{k})} $ \newline $\psi(X_{1}^{(m_{1})}, Y_{1}^{(n_{1})}, X_{2}^{(m_{2})}, Y_{2}^{(n_{2})}, \ldots X_{k}^{(m_{k})}, Y_{k}^{(n_{k})}, A_{1}^{(l_{1})}, A_{2}^{(l_{2})}, \ldots A_{j}^{(l_{j})}) $

\bigskip

This can be re-written as

\bigskip

\noindent $\exists f_{1} \exists f_{2} \cdots \exists f_{k} \forall X_{1}^{(m_{1})} \forall X_{2}^{(m_{2})} \cdots \forall X_{k}^{(m_{k})}$ \newline $\psi(X_{1}^{(m_{1})}, f_{1}(X_{1}^{(m_{1})}), X_{2}^{(m_{2})}, f_{2}(X_{1}^{(m_{1})}, X_{2}^{(m_{2})}), \ldots X_{k}^{(m_{k})}, f_{k}(X_{1}^{(m_{1})}, X_{2}^{(m_{2})}, \ldots X_{k}^{(m_{k})}),$ \newline $A_{1}^{(l_{1})}, A_{2}^{(l_{2})}, \ldots A_{j}^{(l_{j})})$

\bigskip

The point is that if this formula, without the existential function quantifiers, is conceived of as holding in the structure $\langle V_{\kappa}, V_{\kappa+1}, \ldots V_{\kappa+l}, \in, f_{1}, \ldots f_{k}, A_{1}^{(l_{1})}, A_{2}^{(l_{2})}, \ldots A_{j}^{(l_{j})} \rangle$, where $l=\mathrm{max}(m_{1}, n_{1}, \ldots m_{k}, n_{k}, l_{1}-1, \ldots l_{j}-1)-1$, and we try to reflect down to the structure $\langle V_{\beta}, V_{\beta+1}, \ldots V_{\beta+l}, \in, f_{1}^{\beta}, \ldots f_{k}^{\beta}, A_{1}^{(l_{1}),\beta}, A_{2}^{(l_{2}),\beta}, \ldots A_{j}^{(l_{j}),\beta} \rangle$ for some $\beta<\kappa$, then the functions $f_{i}^{\beta}$ are no longer necessarily single-valued. This consideration suggests the following reflection principle.

\begin{Definition} We define $l(\gamma)=\gamma-1$ if $\gamma<\omega$ and $l(\gamma)=\gamma$ otherwise. We extend the definition $A^{(m+1),\beta}=\{B^{(m),\beta}\mid B^{(m)}\in A^{(m+1)}\}$ to $A^{(\alpha),\beta}=\{B^{\beta}\mid B\in A^{(\alpha)}\}$ for all ordinals $\alpha>0$, it being understood that if $V_{\kappa}$ is the domain of discourse then $A^{(\alpha)}$ ranges over $V_{\kappa+l(\alpha)}$. \end{Definition}

\begin{Definition} \label{reflective} Suppose that $\alpha, \kappa$ are ordinals such that $0<\alpha<\kappa$ and that\newline

\noindent (1) $S=\langle \{V_{\kappa+\gamma}\mid \gamma<\alpha\}, \in, f_{1}, f_{2}, \ldots f_{k}, A_{1}, A_{2}, \ldots A_{n} \rangle$ is a structure where each $f_{i}$ is a function $V_{\kappa+l(\gamma_{1})}\times V_{\kappa+l(\gamma_{2})}\times \ldots V_{\kappa+l(\gamma_{i})} \rightarrow V_{\kappa+\zeta_{i}}$ for some ordinals $\gamma_{1}, \gamma_{2}, \ldots \gamma_{i}, \zeta_{i}$ such that $l(\gamma_{1}), l(\gamma_{2}), \ldots l(\gamma_{i}), \zeta_{i} <\alpha$, and each $A_{i}$ is a subset of $V_{\kappa+\delta_{i}}$ for some $\delta_{i}<\alpha$ \newline
(2) $\varphi$ is a formula true in the structure $S$, of the form \newline
$\forall X_{1}^{(\gamma_{1})} \forall X_{2}^{(\gamma_{2})} \cdots \forall X_{k}^{(\gamma_{k})}$ \newline $\psi(X_{1}^{(\gamma_{1})}, f_{1}(X_{1}^{(\gamma_{1})}), X_{2}^{(\gamma_{2})}, f_{2}(X_{1}^{(\gamma_{1})}, X_{2}^{(\gamma_{2})}), \ldots X_{k}^{(\gamma_{k})}, f_{k}(X_{1}^{(\gamma_{1})}, X_{2}^{(\gamma_{2})}, \ldots X_{k}^{(\gamma_{k})}),$ \newline $A_{1}, A_{2}, \ldots A_{j})$ with $\psi$ a formula with first-order quantifiers only \newline
(3) there exists a $\beta$ such that $\alpha<\beta<\kappa$ and a mapping $j:V_{\beta+\alpha}\rightarrow V_{\kappa+\alpha}$, such that $j(X)\in V_{\kappa+\gamma}$ whenever $X\in V_{\beta+\gamma}$, $j(X)=X$ for all $X\in V_{\beta}$, and $j(X)\in j(Y)$ whenever $X\in Y$, and such that, in the structure \newline $S^{\beta}=\langle V_{\beta}, \{V_{\beta+\gamma}\mid 0<\gamma<\alpha\}, \{V_{\kappa+\gamma}\mid 0<\gamma<\alpha\}, \in, j, f_{1}, f_{2}, \ldots f_{k}, A_{1}, A_{2} \ldots A_{n} \rangle$, with variables of order $\gamma$ ranging over $V_{\beta+l(\gamma)}$, we have \newline

\noindent $\forall X_{1}^{(\gamma_{1})} \forall X_{2}^{(\gamma_{2})} \cdots \forall X_{k}^{(\gamma_{k})}$ \newline $\psi(j(X_{1}^{(\gamma_{1})}), f_{1}(j(X_{1}^{(\gamma_{1})})), j(X_{2}^{(\gamma_{2})}), f_{2}(j(X_{1}^{(\gamma_{1})}), j(X_{2}^{(\gamma_{2})})), \ldots j(X_{k}^{(\gamma_{k})}), \newline f_{k}(j(X_{1}^{(\gamma_{1})}), j(X_{2}^{(\gamma_{2})}), \ldots j(X_{k}^{(\gamma_{k})})), A_{1}, A_{2}, \ldots A_{n})$ \newline

\noindent Then we say that the formula $\varphi$ with parameters $A_{1}, A_{2}, \ldots A_{n}$ reflects down from $S$ to $\beta$. If for all structures $S$ of the above form and for all formulas $\varphi$ of the above form true in the structure $S$, this occurs for some $\beta<\kappa$, then $\kappa$ is said to be $\alpha$-reflective. \end{Definition}

It is not clear whether it should be said that the existence of $\alpha$-reflective cardinals follows from the iterative conception of set, because the definition involves a function $j$ which guides the reflection. The idea that the existence of indescribable cardinals follows from the iterative conception of set can be motivated by an idea of Tait \cite{Tait2005a} which Koellner has called the Relativised Cantorian Principle. Cantor wrote that if an initial segment of the sequence of ordinals is only a set then it has a least strict upper bound. The phrase ``is only a set" can be replaced with other conditions for the existence of a least strict upper bound, and for any given set of conditions it then becomes plausible to postulate the existence of a level of the universe which is a closure point for the process of obtaining new ordinals in this way. The indescribable cardinals can then be motivated by the idea that if a level of the universe is describable then it cannot be all of $V$, and so this is a reasonable condition for the existence of a least strict upper bound of all the ordinals obtained so far, and it is reasonable to postulate the existence of a level of the universe which is a closure point for the process of obtaining new ordinals in this way. So whether or not one should similarly admit the existence of reflective cardinals as defined above depends on whether or not one thinks it reasonable for the nonexistence of a function $j$ guiding the reflection of the formula is a sufficient reason to think that the level of the universe obtained so far is not all of $V$, and whether it is reasonable to postulate the existence of a level of the universe which is a closure point for the process of obtaining new ordinals in this way. This may seem doubtful. The Relativised Cantorian Principle always runs the risk of proving too much.

\bigskip

There is however another way to motivate a justification for these cardinals. In \cite{Hellman94}, Hellman discusses the notion of a level of the universe $V_{\alpha}$ ``Putnam-satisfying" a higher-order formula with parameters, with respect to a particular valuation of the free variables. The definition is by induction on the complexity of the formula, with all the usual Tarski clauses, except that we say that $V_{\alpha}$ Putnam-satisfies a formula $\varphi$ with respect to a valuation $E$ of the free variables, starting with an existential quantifier $\exists X$, where $X$ may be a variable of first or higher order, if there exists a $\beta>\alpha$ such that $V_{\beta}$ Putnam-satisfies $\varphi$ with the initial existential quantifier deleted, with respect to some valuation $E'$ agreeing with $E$ at every variable except $X$. (Actually, Hellman makes the further requirement that $\alpha$ and $\beta$ be inaccessible, but we shall not include this requirement.) The levels $V_{\alpha}$ with $\alpha$ inaccessible all agree on what first-order formulas they Putnam-satisfy, but not for the higher-order formulas. 

\bigskip

Suppose we adopt a modified version of Putnam-satisfaction where we start with the standard notion of satisfaction for first-order formulas, and for formulas with higher-order quantifiers we define the notion by induction on the complexity of the formula as before. If one postulates a reflection principle whereby if a level of the universe $V_{\kappa}$ satisfies a higher-order formula with parameters then a lower level Putnam-satisfies it in this modified sense with the same parameters, then one can then proceed to prove the existence (assuming the axiom of choice) of the reflective cardinals discussed here. The mapping $j$ can be constructed (assuming choice) from the Skolem functions witnessing the Putnam-satisfaction, at the lower level, of the formula. The image of the range of the Skolem functions under the rank function can be forced to be bounded above by $\kappa$ by including all of $V_{\kappa+\alpha}$, with $\alpha$ sufficiently large, as a parameter, and modifying the formula to be reflected so as to include a clause saying that all variables in the range of the Skolem functions must be contained in this parameter. Conversely it is easy to see that our reflective cardinals are ``Putnam-reflective" in the sense just defined. This might be thought to be a somewhat more compelling justification.

\bigskip

We now give a consistency proof for this large cardinal property.

\begin{theorem} Suppose that $0<\alpha<\kappa$ and $\kappa$ is $\beth_{\kappa+\alpha+1}$-supercompact. Then $\kappa$ is $\alpha$-reflective. \end{theorem}

\begin{proof} Suppose that $0<\alpha<\kappa$ and $\kappa$ is $\beth_{\kappa+\alpha+1}$-supercompact. Then there exists an elementary embedding $k:V\rightarrow M$ with critical point $\kappa$ such that $^{\beth_{\kappa+\alpha+1}}{M}\subset M$. Let $S=\langle \{V_{\kappa+\gamma}\mid \gamma<\alpha\}, \in, f_{1}, f_{2}, \ldots f_{k}, A_{1}, A_{2}, \ldots A_{n} \rangle$ be a structure and $\varphi$ a formula as in the definition of an $\alpha$-reflective cardinal. Working in $M$, consider the structure $k(S)$. Since $^{\beth_{\kappa+\alpha+1}}{M}\subset M$, the elementary embedding $k$ induces a mapping $j\in M$ as in the definition of an $\alpha$-reflective cardinal such that the structure $k(S)$ reflects down to $\kappa$ in $M$ via the mapping $j$. Since $k$ is an elementary embedding we may infer that there exists a $\delta<\kappa$ such that $S$ reflects down to $\delta$ in $V$. \end{proof}

\bigskip

Next we show that $\alpha$-reflective cardinals relativise to the constructible universe $L$.

\bigskip

\begin{theorem} Suppose that $\alpha<\kappa$ and $\kappa$ is $\alpha$-reflective. Then $\kappa$ is $\alpha$-reflective in the constructible universe $L$. \end{theorem}

\begin{proof} Suppose that $\alpha<\kappa$. Let $S^{L}=\langle \{V_{\kappa+\gamma}^{L}\mid \gamma<\alpha\}, \in, f_{1}, f_{2}, \ldots f_{k}, A_{1}, A_{2}, \ldots A_{n} \rangle\in L$ be a structure in $L$ as in Definition \ref{reflective}, and let $\varphi$ be a formula as in Definition \ref{reflective} that is true in $S^{L}$. We may consider the formula $\varphi^{L}$ with all $\gamma$-order quantifiers, for $\gamma<\alpha$, relativized to $V_{\kappa+l(\gamma)}^{L}$. Then $\varphi^{L}$ holds in $S=\langle\{V_{\kappa+\gamma}\mid\gamma<\alpha\},\in, f_{1}, f_{2}, \ldots f_{k}, A_{1}, A_{2}, \ldots A_{n}\rangle$. By introducing new Skolem functions for $\varphi^{L}$ into the structure $S$ to produce an expanded structure $S'$, we may replace $\varphi^{L}$ with a formula $\psi$ which is true in the expanded structure $S'$. Then since $\kappa$ is $\alpha$-reflective in $V$ then there must be a mapping $j$ which witnesses that $\psi$ reflects down to some $\beta<\kappa$. One can now make use of the canonical well-ordering of $L$. For each member of the well-ordering in the domain of $j\mid L$, one picks for the value of $j$ at this point the least member of the well-ordering which is "admissible" in the sense that the function $j$ constructed so far has an extension to the entire domain of $j\mid L$ with the desired properties. Since the well-ordering of $L$ is definable, in this way one can ensure that $j\mid_{L}\in L$. This shows that $\varphi$ reflects down from $S^{L}$ to $\beta$ in $L$. \end{proof}

\bigskip

Next we show that these cardinals are consistent relative to $\kappa(\omega)$.

\bigskip

\begin{theorem} Suppose that $\kappa(\omega)$ exists. Then the set of all $\lambda<\kappa(\omega)$ such that $\lambda$ is $\alpha$-reflective for all $\alpha$ such that $0<\alpha<\lambda$ is a stationary subset of $\kappa(\omega)$. \end{theorem}

\begin{proof} Suppose that $\kappa=\kappa(\omega)$. Let $C$ be a closed unbounded subset of $\kappa$. We must show that there is a cardinal $\lambda\in C$ such that $\lambda$ is $\alpha$-reflective for all $\alpha$ such that $0<\alpha<\lambda$. Let $S=\{\iota_{1}, \iota_{2}, \ldots \}$ be a set of Silver indiscernibles for the structure $\langle V_{\kappa},\epsilon,C \rangle$. It can be shown that it is possible to find such a set such that all the $\iota_{i}$ are in $C$. Let $M$ be a Skolem hull of $S$ in this structure and let $\lambda=\iota_{2}$. Then the mapping $\iota_{k}\mapsto\iota_{k+1}$ induces an elementary embedding $j:M\rightarrow M$. Fix a well-ordering of $(V_{\lambda})^{M}$ in $M$ and define a set of Skolem functions with respect to this well-ordering. If we let $M_{n,n'}$ be the Skolem hull of $(V_{\lambda})^{M}\cup \{\iota_{1},\iota_{2},\ldots \iota_{n'}\}$ in $M$ for Skolem terms for formulas of complexity no greater than $\Sigma_{n}$, then the mapping $\iota_{k}\mapsto\iota_{k+1}$ induces a mapping $j_{n,n'}:M_{n,n'}\rightarrow M_{n,n'+1}$ which respects the Skolem functions for formulas of complexity no greater than $\Sigma_{n}$, and restrictions of this mapping to intersections of sets of the form $(V_{\lambda+\alpha})^{M}$ with the domain of the mapping - which we call truncations of the mapping - are members of $M$, because they are definable with parameters from $M$. If $\varphi$ is a formula with parameters $A_{1}, A_{2}, \ldots A_{m}$ as in the definition of an $\alpha$-reflective cardinal where $0<\alpha<\lambda$ then $\varphi$ will reflect down in $M$ from $\iota_{2}$ to $\iota_{1}$ by means of a truncation of the mapping $j_{n,n'}$ where $n$ and $n'$ are sufficiently large. This shows that $\iota_{2}$ is $\alpha$-reflective in $M$ and hence in $V$. \end{proof}

Next we establish some properties of $\omega$-reflective cardinals.

\section{Restricted versions of Tait's reflection principles}

In \cite{Tait2005a} Tait defines the following set of reflection principles.

\begin{Definition} A formula in the $n$th-order language of set theory, some $n<\omega$, is positive iff it is built up by means of the operations $\vee$, $\wedge$, $\forall$, $\exists$ from atoms of the form $x=y$, $x \neq y$, $x \in y$, $x \notin y$, $x \in Y^{(2)}$, $x \notin Y^{(2)}$ and $X^{(m)} = X'^{(m)}$ and
$X^{(m)} \in Y^{(m+1)}$, where $m \geq 2$. \end{Definition}

\begin{Definition} For $0<n<\omega$, $\Gamma^{(2)}_{n}$ is the class of formulas

\begin{equation} \forall X_{1}^{(2)} \exists Y_{1}^{(k_{1})} \cdots \forall X_{n}^{(2)} \exists Y_{n}^{(k_{n})} \varphi(X_{1}^{(2)}, Y_{1}^{(k_{1})}, \ldots, X_{n}^{(2)}, Y_{n}^{(k_{n})}, A^{(l_{1})}, \ldots A^{(l_{n'})}) \end{equation}

where $\varphi$ is positive and does not have quantifiers or second or higher-order and $k_{1}, \ldots k_{n}, l_{1}, \ldots l_{n'}$ are natural numbers. \end{Definition}

\begin{Definition} We say that $V_{\kappa}$ satisfies $\Gamma^{(2)}_{n}$-reflection if, for all $\varphi\in\Gamma^{(2)}_{n}$, if $V_{\kappa}\models\varphi(A^{(m_{1})},A^{(m_{2})},\ldots A^{(m_{p})})$ then $V_{\kappa}\models\varphi^{\delta}(A^{(m_{1}),\delta},A^{(m_{2}),\delta},\ldots A^{(m_{p}),\delta})$ for some $\delta<\kappa$. \end{Definition}

\begin{theorem} [\textbf{Koellner}] Suppose that $\kappa=\kappa(\omega)$ is the first $\omega$-Erd\"os cardinal. Then there exists a $\delta<\kappa$ such that $V_{\delta}$ satisfies $\Gamma^{(2)}_{n}$-reflection for all $n$. \end{theorem}

\begin{theorem} [\textbf{Tait}] Suppose that $n<\omega$ and $V_{\kappa}$ satisfies $\Gamma^{(2)}_{n}$-reflection. Then $\kappa$ is $n$-ineffable. \end{theorem}

\begin{theorem} [\textbf{Tait}] Suppose that $\kappa$ is measurable. Then $V_{\kappa}$ satisfies $\Gamma^{(2)}_{n}$-reflection for all $n<\omega$. \end{theorem}

In \cite{Tait2005a} Tait proposes to define $\Gamma^{(m)}_{n}$ in the same way as the class of formulas $\Gamma^{(2)}_{n}$, except that universal quantifiers of order $\leq m$ are permitted. Koellner shows in \cite{Koellner2009} that this form of reflection is inconsistent when $m>2$. We formulate a new form of reflection which we will be able to prove holds for an $\omega$-reflective cardinal.

\begin{Definition} For $2\leq m<\omega$, $0<n<\omega$, $\Gamma^{*(m)}_{n}$ is the class of formulas

\begin{equation} \label{gammaequation} \forall X_{1}^{(k_{1})} \exists Y_{1}^{(l_{1})} \cdots \forall X_{n}^{(k_{n})} \exists Y_{n}^{(l_{n})} \psi(X_{1}^{(k_{1})}, Y_{1}^{(l_{1})}, \ldots, X_{n}^{(k_{n})}, Y_{n}^{(l_{n})}, A^{(m_{1})}, \ldots A^{(m_{p})}) \end{equation}

where $\psi$ is positive and does not have quantifiers or second or higher-order and \newline $k_{1}, \ldots k_{n}, l_{1}, \ldots l_{n}, m_{1}, \ldots m_{p}$ are natural numbers such that $l_{j} \geq k_{i}$ whenever $0<i\leq j \leq n$. \end{Definition}

\begin{Definition} We say that $V_{\kappa}$ satisfies $\Gamma^{*(m)}_{n}$-reflection if, for all $\varphi\in\Gamma^{*(m)}_{n}$, if $V_{\kappa}\models\varphi(A^{(m_{1})},A^{(m_{2})},\ldots A^{(m_{p})})$ then $V_{\kappa}\models\varphi^{\delta}(A^{(m_{1}),\delta},A^{(m_{2}),\delta},\ldots A^{(m_{p}),\delta})$ for some $\delta<\kappa$. \end{Definition}

We shall now prove that if $\kappa$ is $\omega$-reflective then $V_{\kappa}$ satisfies $\Gamma^{*(m)}_{n}$-reflection for all $m\geq 2, n>0$. Note that $\Gamma^{*(2)}_{n}$-reflection is the same as $\Gamma^{(2)}_{n}$-reflection.

\begin{theorem} Suppose that $\kappa$ is $\omega$-reflective. Then $V_{\kappa}$ satisfies $\Gamma^{*(m)}_{n}$-reflection for all $m\geq 2, n>0$. \end{theorem}

\begin{proof} Suppose that $\varphi\in\Gamma^{*(m)}_{n}$ is true in $\langle\{V_{\kappa+n}\mid n\in\omega\},\in, \ldots\rangle$ and that $\varphi$ is as in (2). There must exist functions $f_{1}, f_{2}, \ldots f_{n}$ such that

\begin{equation} \forall X_{1}^{(k_{1})} \ldots \forall X_{n}^{(k_{n})} \psi(X_{1}^{(k_{1})}, f_{1}(X_{1}^{(k_{1})}), \ldots X_{n}^{(k_{n})}, f_{n}(X_{1}^{(k_{1})}, X_{2}^{(k_{2})}, \ldots X_{n}^{(k_{n})}), A^{(m_{1})}, \ldots A^{(m_{p})}) \end{equation}

is true in $\langle \{V_{\kappa+n}\mid n\in\omega\}, \in, \ldots \rangle$. Since $\kappa$ is $\omega$-reflective there will be some $\beta<\kappa$ and a function $j:V_{\beta+\omega}\rightarrow V_{\kappa+\omega}$ as in Definition \ref{reflective} such that

\begin{equation} \forall X_{1}^{(k_{1})} \ldots \forall X_{n}^{(k_{n})} \psi(j(X_{1}^{(k_{1})}), f_{1}(j(X_{1}^{(k_{1})})), \ldots j(X_{n}^{(k_{n})}), f_{n}(j(X_{1}^{(k_{1})}), j(X_{2}^{(k_{2})}), \ldots \end{equation} $j(X_{n}^{(k_{n})})), A^{(m_{1})}, \ldots A^{(m_{p})})$

\bigskip

\noindent is true in $S=\langle\{V_{\beta+n}\mid n\in\omega\},\in,\ldots\rangle$. As Koellner observes in \cite{Koellner2003}, when $k_{i}=2$ for each $i$ this is enough to prove $\Gamma^{(2)}_{n}$-reflection. This is because the map $X^{(2)}\mapsto j(X^{(2)}\cap V_{\beta}$ is surjective (in fact, the identity) on $V_{\beta+1}$, and since $\psi$ is positive the truth-value of the formula in $S$, for a particular valuation of the free variables, never passes from true to false on relativisations of all the free variables and parameters (and the relativisations of the Skolem functions are still single-valued). To establish $\Gamma^{*(m)}_{n}$-reflection for $m>2$, we replace $j$ in the above formula with a function $j'$ which agrees with $j$ on $V_{\beta+1}$, and satisfies $j'(X)=\{j'(Y)\mid Y\in X\}$ on $V_{\beta+k}\setminus V_{\beta+k-1}$, for $k=2, \ldots m$. Now the mapping $X^{(k)}\mapsto j(X^{(k)})\cap V_{\beta+k-1}$ is the identity for $k=2, \ldots m$. We also replace all the parameters $A^{(m_{i})}$ with paremeters $A'^{(m_{i})}$, chosen to as to be the same as the original paremeters except that every element of the transitive closure of the form $j(X)$ is replaced by $j'(X)$. Because $l_{j}\geq k_{i}$ for $0<i\leq j<n$, it is now possible to modify the Skolem functions in such a way that all the membership relations between images of free variables under $j'$, parameters, and images of free variables under $f_{i}\circ j'$ where $f_{i}$ is one of the Skolem functions, will be the same as they were before when we were using the old parameters and $j$ instead of $j'$. Hence the formula (4) will be true in $S$, and then the result of relativising all the parameters and free variables in the formula to $\beta$ will be true in $\{V_{\beta+n}\mid n\in\omega\}, \in, \ldots \rangle$. This completes the proof. \end{proof}

It is also easy to see by examining Koellner's proofs in \cite{Koellner2009} that $\omega$-reflective cardinals satisfy the reflection principles which he proves consistent there.

\bigskip

It is plausible to regard $\alpha$-reflective cardinals as the natural generalization of Tait's proposed reflection principles. The fact that they do not break the $\kappa(\omega)$ barrier provides further evidence for the view that Koellner has expressed in \cite{Koellner2009} that no reflection principle does so, and reflection principles are not sufficient to effect a significant reduction in incompleteness of ZFC.

\bib


\begin{thebibliography}{1}

\bibitem{Hellman94}
Geoffrey Hellman.
\newblock {\em Mathematics without numbers: Towards a modal-structural
  interpretation}.
\newblock Oxford University Press, 1994.

\bibitem{Koellner2003}
Peter Koellner.
\newblock {\em The Search for New Axioms}.
\newblock PhD thesis, Massachusetts Institute of Technology, 2003.

\bibitem{Koellner2009}
Peter Koellner.
\newblock On reflection principles.
\newblock {\em Annals of Pure and Applied Logic}, (157), 2009.

\bibitem{Reinhardt74}
W.~Reinhardt.
\newblock Remarks on reflection principles, large cardinals, and elementary
  embeddings. {I}n: Proceedings of symposia in pure mathematics. vol. 10.
\newblock pages 189--205.

\bibitem{Tait2005a}
William Tait.
\newblock {\em Constructing Cardinals from Below: In \cite{Tait2005b}}, pages
  133--154.
\newblock Oxford University Press, 2005.

\bibitem{Tait2005b}
William Tait.
\newblock {\em The Provenance of Pure Reason: Essays In the Philosophy of
  Mathematics and Its History}.
\newblock Oxford University Press, 2005.

\end{thebibliography}
\end{document}